\documentclass[10pt]{article}
\usepackage{graphicx,subfigure}
\usepackage{amsmath,amsthm}
\usepackage[american]{babel}
\usepackage[utf8]{inputenc}
\usepackage[T1]{fontenc}
\usepackage[per-mode=symbol,detect-weight]{siunitx}
\usepackage{amssymb}
\usepackage{epstopdf}
\usepackage{a4wide}

\usepackage[vlined,ruled,linesnumbered]{algorithm2e}
\usepackage{subfigure}
\newtheorem{definition}{Definition}[section]
\newtheorem{lemma}{Lemma}[section]
\newtheorem{remark}{Remark}[section]

\newtheorem{pro}{Proposition}[section]
\newtheorem{theorem}{Theorem}[section]
\newtheorem{corollary}{Corollary}[theorem]
\usepackage[justification=centering]{caption}
\usepackage{mathtools}
\usepackage[labelfont=bf, labelsep=newline, justification=raggedright,
   singlelinecheck=false]{caption}
 \usepackage{graphics, graphicx}
\usepackage{amsfonts}
\DeclarePairedDelimiter{\norm}{\ lVert}{\ rVert}
\usepackage{commath}
\usepackage{enumerate}
\DeclareFontFamily{OT1}{pzc}{}
\DeclareFontShape{OT1}{pzc}{m}{it}{<-> s * [1.10] pzcmi7t}{}
\DeclareMathAlphabet{\mathpzc}{OT1}{pzc}{m}{it}

\usepackage{url}

\begin{document}
	\begin{center}
		{{\bf \large {\rm {\bf Existence and uniqueness of time-fractional diffusion equation on a metric star graph}}}}
	\end{center}
	\begin{center}
		{\textmd {{\bf Vaibhav Mehandiratta,}}}\footnote{\it Department of Mathematics, Indian Institute of Technology, Delhi, India, (vaibhavmehandiratta@gmail.com)}
		{\textmd {{\bf Mani Mehra,}}}\footnote{\it Department of Mathematics, Indian Institute of Technology, Delhi, India, {(mmehra@maths.iitd.ac.in)} }
		{\textmd {{\bf G\"{u}nter Leugering}}}\footnote{\it Friedrich-Alexander-Universit\"{a}t Erlangen-N\"{u}rnberg (FAU), Lehrstuhl Angewandte Mathematik II, Cauerstr. 11, 91058 Erlangen, Germany {(guenter.leugering@fau.de)} }
	\end{center}
\begin{abstract}
 In this paper, we study the time-fractional diffusion equation on a metric star graph. The existence and uniqueness of the weak solution are investigated and the proof is based on eigenfunction expansions. Some priori estimates and regularity results of the solution are proved.
\end{abstract}
\numberwithin{equation}{section}
\section{Introduction}
We consider a graph $\mathcal{G}=(V,E)$ consisting of a finite set of vertices (nodes) $V=\{v_i:i=0,1,2,...,k\}$ and a set of edges $E$ (such as heat conducting elements) connecting these nodes. The graph considered in this work is a metric graph \cite{Mug14}. Therefore, each edge $e_i$, $i=1,2,\dotsc,k$ is parametrised by an interval $(0,l_i)$. The study of partial differential equations (PDEs) on networks or metric graphs is not just the analysis of known mathematical objects on special domains, since in our context, graphs or networked domains are not manifolds. Thus, we investigate PDEs on single edges of graph (interpreted as continuous curves or manifolds) \cite{shukla2019fast} along with certain transmission conditions such as continuity and Kirchoff condition at junction node. Hence, we define a coordinate system on each edge $e_i$ by taking $v_0$ as the origin and $x\in (0,l_i)$ as the coordinate. We consider a time-fractional diffusion equation on a metric star graph $\mathcal{G}$, which is a graph consisting of $k$ edges incident to a common vertex $v_0$ (see Figure $\ref{star}$):
\begin{align}
&_{C}D^{\alpha}_{0,t}y(x,t)=\frac{\partial^2y(x,t)}{\partial x^2}+f(x,t),\quad x\in \mathcal{G}, ~t\in (0,T), ~0<\alpha<1.\label{graphmodel}\\
&y(x,0)=y^0(x),\quad x\in \mathcal{G}.\label{graphini}
\end{align}
More precisely, at each edge we have following fractional diffusion equation
\begin{align}
&_{C}D^{\alpha}_{0,t}y_i(x,t)=\frac{\partial^2y_i(x,t)}{\partial x^2}+f_i(x,t),\quad x\in (0,l_i), ~t\in (0,T), ~0<\alpha<1,\label{edgemodel}\\
&y_i(x,0)=y^0_i(x),\quad x\in (0,l_i),\quad i=1,2\dotsc,k,\label{edgeini}
\end{align}
along with the continuity and Kirchoff conditions at junction node $v_0$ as
\begin{align}
&y_i(0,t)=y_j(0,t),~ i\neq j,~ i,j=1,2,\dotsc,k,~t\in (0,T),\label{edgecontinuity}\\
&\sum_{i=1}^{k}\frac{\partial y_i(0,t)}{\partial x}=0,\label{edgekirchoff}
\end{align}
and Dirichlet boundary conditions at boundary nodes $v_i$
\begin{equation}\label{edgeboundary}
y_i(l_i,t)=0,\quad t\in (0,T).
\end{equation}
Here $_{C}D^{\alpha}_{0,t}$ denotes the Caputo fractional derivative of order $\alpha$ with respect to $t$ defined as
\[_{C}D^{\alpha}_{0,t}y(x,t)=\frac{1}{\Gamma(1-\alpha)}\left(\int_{0}^{t}(t-\xi)^{-\alpha}\frac{\partial y(x,\xi)}{\partial \xi}d\xi\right),\quad 0<\alpha<1,\quad t\in (0,T),\]
where $\Gamma(.)$ denotes the Euler gamma function. In this paper we prove the existence and uniqueness of the weak solution of initial-value problem (IVP)  $(\ref{graphmodel})$-$(\ref{graphini})$ whose restriction to the edge $e_i$ gives the weak solution of initial-boundary value problem (IBVP) $(\ref{edgemodel})$-$(\ref{edgeboundary})$. When $\alpha$ approaches 1, the Caputo fractional derivative $_{C}D^{\alpha}_{0,t}u$ approaches the ordinary derivative $\frac{\partial y}{\partial t}$ and, thus, IBVP $(\ref{edgemodel})$-$(\ref{edgeboundary})$ represents the standard diffusion equation on graphs for which existence and uniqueness was proved in \cite{Max18}. Recently in \cite{mehandiratta2019existence}, authors established the existence and uniqueness of nonlinear fractional boundary value problem on a star graph. Hence, this work could be seen as the extension of \cite{mehandiratta2019existence} for time dependent problem.\vspace{0.1cm}
\par The origin of the study of differential equation on graphs can be traced back to 1980s with Lumer's work \cite{Lum80} on ramification spaces. In \cite{Nic85}, Nicaise investigated the Propagation of nerve impulses. Since then, considerable work related to eigenvalue problems (Sturm-Liouville type problems) on networks, i.e. metric graphs has been done, for instance see the article by von Below \cite{Von85} and \cite{Pok83, Gera88, Pav83}. Partial differential equations on graphs or multi-link structures plays important role in the field of science and engineering. For instance, the flows on the nets of gas pipeline \cite{Ste07}, controlled vibrations of networks of strings (hyperbolic wave equations) \cite{Leu00}, water wave propagation in open channel networks (Burgers type equation) \cite{Yos14} naturally lead to partial differential equation on graphs. 
\par Evolutionary problems (such as parabolic equations) on metric graphs were considered in \cite{Von88b}. The dynamic networks of strings and beams along with their control properties were studied by Lagnese et al. in \cite{LagneseLeugeringSchmidt1994}, see also e.g. \cite{LagneseLeugeringSchmidt1994b,LagneseLeugeringSchmidt1993,
		LagneseLeugeringSchmidt1993b,LagneseLeugeringSchmidt1992}.
	\begin{figure}
		\centering
		\captionsetup{justification=centering}
		\includegraphics[width=113mm]{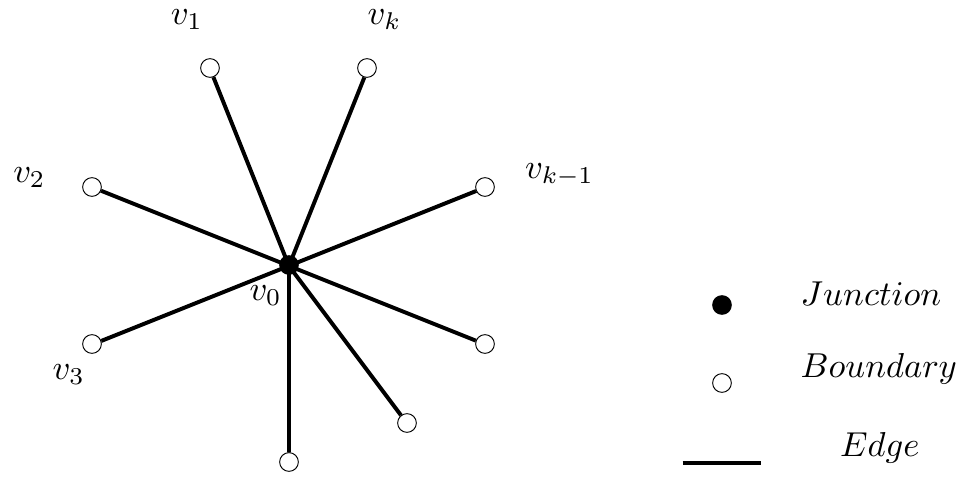}
		\caption{A sketch of the star graph with $k$ edges}
		\label{star}
	\end{figure}
	The progress of problems defined on metric graphs until 2006 has been presented in an excellent survey by Dager and Zuazua \cite{DagerZuazua}. Since then, modeling, analysis and optimal control problems for linear and nonlinear partial differential equations on metric graphs has become an active area of research. Nonlinear Schr\"{o}dinger equation on metric star graph were studied by Adami et al. in \cite{Adami14}. In \cite{Yos14}, Yoshioka et al. considered the Burger type equation models on connected graph and discussed the existence and uniqueness of the model along with the energy estimates. Inverse problem on metric graphs were initiated by Nicaise in \cite{Nic03} for the wave equation. In \cite{Avd14}, Avdonin and Nicaise considered the source identification problem for the wave equation on an interval and extended their approach to study the problem on trees (graphs which do not contain a cycle), while source identification problem for the heat equation on metric graphs was discussed in \cite{avdonin2015source}. Recently, Grigor et al. \cite{Grigor16} studied the Yamabe type equations on graphs and proved the existence of positive solution using the mountainpass theorem due to Ambrosetti-Rabinowich.
	\par 
	On the other hand, fractional calculus find its importance in different fields of science and engineering \cite{Hilf00, Fried92, Magin08, Bohan08}. A strong motivation for the study and analysis of fractional diffusion equations comes from the fact that they efficiently describe the phenomenon of anomalous diffusion \cite{luchko2012anomalous}. Fractional diffusion equations on bounded domains has been studied by various authors. For instance in \cite{Luchko09}, Luchko gave the maximum principle for the time-fractional diffusion equation, while in \cite{Luchko10} he established the existence and uniqueness results for time-fractional diffusion equation using eigenfunction expansion by taking source term $f=0$. In \cite{Moto11}, the existence results for fractional diffusion-wave equations were established by Sakamoto and Yamamoto, while IBVP for a coupled fractional diffusion system was discussed in \cite{li2015existence}. For more results we refer \cite{ Man96, Wei18} and refernces therein.\vspace{0.1cm}
	\par 
	To the best of our knowledge, there has not been any published work related to the existence and uniqueness results for the time-fractional diffusion equation on metric graphs so far. In this paper, we focus on proving the existence and uniqueness of IBVP $(\ref{edgemodel})$-$(\ref{edgeboundary})$ and study the regularity of solution given by the eigenfunction expansions.\vspace{0.1cm}
	\par 
	The rest of the paper is divided into three sections. In Sec 2, we define the function spaces for star graph $\mathcal{G}$, state some propostions regarding Mittag-Leffler function (defined in section 2) and prove a Lemma by means of eigenfunction expansion which plays an important role in developing the detailed analysis of the problem. In Sec. 3, we prove the main results on existence and uniqueness of IBVP $(\ref{edgemodel})$-$(\ref{edgeboundary})$ under different regularity conditions on initial data $y^{0}(x)$. In Sec. 4, we conclude the work done and provide a brief idea of the future direction.
\section{Preliminaries}
First of all, we define the following function spaces on a star graph $\mathcal{G}$:
 \begin{align*}
 &L_2(\mathcal{G})=\displaystyle\prod_{i=1}^{k}L_2(0,l_i),\\
 &H_m(\mathcal{G})=\displaystyle\prod_{i=1}^{k}H_m(0,l_i).
 \end{align*}
 with the corresponding inner products
\[\langle {y,w}\rangle_{L_2(\mathcal{G})}:=\sum_{i=1}^{k}\langle {y_i,w_i}\rangle_{L_2(0,l_i)}\]
and
\[\langle {y,w}\rangle_{H_m(\mathcal{G})}:=\sum_{i=1}^{k}\langle {y_i,w_i}\rangle_{H_m(0,l_i)},\]
where $L_2(0,l_i)$ and $H_m(0,l_i)$ are standard Sobolev spaces. The spaces $L_2(\mathcal{G})$ and $H_m(\mathcal{G})$ are Hilbert spaces with inner products $\langle \cdot ,\cdot \rangle_{L_2(\mathcal{G})}$ and $\langle \cdot, \cdot \rangle_{H_m(\mathcal{G})}$ respectively \cite{Mug07}.\vspace{0.1cm}
\par We define the following operator $\mathcal{L}$ on the Hilbert space $L_2(\mathcal{G})$:
\begin{align*}
D(-\mathcal{L})=&\bigg\{y\in L_2(\mathcal{G}): y_i\in H_2(0,l_i),\\ &\hspace{0.4cm}y_i(l_i)=0,
y_i(0)=y_j(0),~~ i\neq j,~i,j=1,2,\dotsc,k~\text{and}~ \sum_{i=1}^{k}y_i'(0)=0\bigg\},\\
&\forall y\in D(-\mathcal{L}):~\mathcal{L}y=\left(\frac{\partial^{2}y_i}{\partial{x}^{2}}\right)_{i=1}^{k}.
\end{align*}
\begin{remark}
	The operator $\mathcal{-L}$ is a non-negative self-adjoint operator since it is the Friedrichs extension of the triple $(L_2(\mathcal{G}); V; a)$ defined by \cite{Kato80}
	\[V=\bigg\{y\in \displaystyle\prod_{i=1}^{k}H_1(0,l_i): y_i(l_i)=0,
	y_i(0)=y_j(0),~ i\neq j,~i,j=1,2,\dotsc,k\bigg\},\]
	which is a Hilbert space with the inner product
	\[\langle {u,w}\rangle_{V}:=\sum_{i=1}^{k}\langle {y_i,w_i}\rangle_{H_1(0,l_i)}=\sum_{i=1}^{k}\int_{0}^{l_i}y_i'w_i'dx,\]
	and \[a(y,w)=\sum_{i=1}^{k}\int_{0}^{l_i}y_i'(x)w_i'(x)dx.\]
\end{remark}
The spectrum of operator $\mathcal{-L}$ consist of eigenvalues, having the form
\[0< \mu_1(\mathcal{G})\leq \mu_2(\mathcal{G})\leq \dotsc \to \infty;\]
 and the eigenfunction $\Psi_n=(\psi_{n,1}, \psi_{n,_2},\dotsc, \psi_{n,_k})$ corresponding to eigenvalue $\mu_n$: $\mathcal{-L}\Psi_n=\mu_n\Psi_n$, $n\in \mathbb{N}$. Then the sequence $\{\Psi_n\}_{n\in \mathbb{N}}$ forms an orthonormal basis of $L_2(\mathcal{G})$ (see \cite{Prov08, Nic85}). Hence $\{\mu_n, \Psi_n\}_{n\in\mathbb{N}}$ is the eigensystem of following problem:
 \begin{align}
 &\psi^{''}_{n,i}(x)=-\mu_n\psi_{n,i}(x),\quad 0<x<l_i,\label{eigproblem}\\
 &\psi_{n,i}(l_i)=0,~ i=1,2,\dotsc,k, \label{boundary}\\
 &\psi_{n,i}(0)=\psi_{n,j}(0),~ i,j=1,2,\dotsc,k,~i\neq j, \label{continuity}\\
 &\sum_{i=1}^{k}\psi_{n,i}'(0)=0.\label{kirchoff} 
\end{align}
Now, we define the fractional power $(\mathcal{-L})^{\gamma}$, $\gamma\in \mathbb{R}$ using the spectral decompostion of operator $\mathcal{L}$. For any $y\in L_2(\mathcal{G})$, we have
\[y=\sum_{n=1}^{\infty}\langle y,\Psi_n\rangle \Psi_n,\]
which gives \[y_i=\sum_{n=1}^{\infty}\langle y,\Psi_n\rangle \psi_{n,i},\]
where $y_i$ is the restriction of $u$ to the edge $e_i$.
\par Hence we define $(\mathcal{-L})^{\gamma}y=\left((-\mathcal{M})^{\gamma}y_i\right)_{i=1}^{k}$,~ where $(\mathcal{M})^{\gamma}y_i=\displaystyle\sum_{n=1}^{\infty}\mu_n^{\gamma}\langle y,\Psi_n\rangle\psi_{n,i}$.
Now \begin{align*}
\norm{(\mathcal{-L})^{\gamma}y}_{L_2(\mathcal{G})}^{2}&=\sum_{i=1}^{k}\norm{(-\mathcal{M})^{\gamma}y_i}_{L_2(0,l_i)}^{2}\\
&=\sum_{n=1}^{\infty}\mu_n^{2\gamma}\left|\langle y,\Psi_n\rangle\right|^{2}.
\end{align*}
Then we define
\[D\left((\mathcal{-L})^{\gamma}\right)=\left\{y\in L_2(\mathcal{G}):~\sum_{n=1}^{\infty}\mu_n^{2\gamma}\left|\langle y,\Psi_n\rangle\right|^{2}<\infty\right\}.\]
It follows that $D\left((\mathcal{-L})^{\gamma}\right)$ is a Hilbert space with the norm
\begin{equation}\label{dlnorm}
\norm{y}_{D((\mathcal{-L})^{\gamma})}=\norm{(\mathcal{-L})^{\gamma}y}_{L_2(\mathcal{G})}=\left(\sum_{n=1}^{\infty}\mu_n^{2\gamma}\left|\langle y,\Psi_n\rangle\right|^{2}\right)^{\frac{1}{2}}.\end{equation}
\begin{remark}
Using Parseval's identity, we have \cite{} 
\[\norm{u}_{V}^2\sim \norm{u}_{D(\mathcal{-L}^{1/2})}^2,\]
while in general $D\left((\mathcal{-L})^{\gamma}\right)\subset H_{2\gamma}(\mathcal{G})$ holds for $\gamma>0.$. Hence, in view of $(\ref{dlnorm})$, the spaces $V$, $L_2(\mathcal{G})$ and $H_2(\mathcal{G})$ can be characterised as follows:
\[V=\left\{y=\sum_{n=1}^{\infty}\langle y,\Psi_n\rangle \Psi_n:\quad \|{y}\|_{V}^2=\sum_{n=1}^{\infty}\mu_n\left|\langle y,\Psi_n\rangle\right|^2<\infty\right\},\]
 \[L_2(\mathcal{G})=\left\{y=\sum_{n=1}^{\infty}\langle y,\Psi_n\rangle \Psi_n:\quad \|{y}\|_{L_2(\mathcal{G})}^2=\sum_{n=1}^{\infty}\left|\langle y,\Psi_n\rangle\right|^2<\infty\right\}\]
 and
  \[H_2(\mathcal{G})=\left\{y=\sum_{n=1}^{\infty}\langle y,\Psi_n\rangle \Psi_n:\quad \|{y}\|_{H_2(\mathcal{G})}^2=\sum_{n=1}^{\infty}\mu_n^{2}\left|\langle y,\Psi_n\rangle\right|^2<\infty\right\}.\]
\end{remark}
\par
Now, we give the following definition and propositions regarding Mittag-Leffler function which will be used later.
\begin{definition} The Mittag-Leffler function is defined as follows\\
\par\hspace{0.5cm}
$E_{\alpha,\beta}(z)=\displaystyle\sum_{j=0}^{\infty}\frac{z^{j}}{\Gamma{(\alpha j+\beta)}},\quad z\in \mathbb{C}$,\vspace{0.2cm}\\
where $\alpha>0$ and $\beta\in \mathbb{R}$ are arbitrary constants.
\end{definition}
\begin{pro}[see \cite{Pod99}]\label{boundmit}
Let $0<\alpha< 2$, $\beta\in \mathbb{R}$ be arbitrary and $\mu$ be such that $\pi \alpha/2<\mu<min\{\pi, \pi \alpha\}$, then there exists a constant $C=C(\alpha, \beta, \mu)>0$ such that\\
\par\hspace{0.5cm}
 $\left|E_{\alpha,\beta}(z)\right|\leq \frac{C}{1+|z|}, \quad \mu\leq|arg(z)|\leq \pi$.
\end{pro}
\begin{pro}\label{monotone2}
	Let $0<\alpha< 1$ and $\eta>0$, then we have $0<E_{\alpha, \alpha}(-\eta)<\frac{1}{\Gamma{(\alpha)}}$. Furthermore, $E_{\alpha,1}(-\eta)$ is a monotonic decreasing function with $\eta>0$.
\end{pro}
\begin{pro}[see \cite{Poll48}]\label{monotone1}
	Let $0<\alpha< 1$ and $t>0$, then we have $0<E_{\alpha, 1}(-t)<1$. Furthermore, $E_{\alpha,1}(-t)$ is completely monotonic that is\\
	\par\hspace{0.2cm}
	$(-1^{n})\frac{d^n}{dt^n}E_{\alpha,1}(-t)\geq 0$,\quad $n\in \mathbb{N}$.
\end{pro}
\begin{pro}[see \cite{Moto11}]\label{dermit}
Let $\alpha>0$, $\lambda>0$ and $m\in \mathbb{N}$, then we have\\
\par\hspace{0.2cm}
 $\frac{d^m}{dt^m}E_{\alpha,1}(-\lambda t^{\alpha})=-\lambda t^{\alpha-m}E_{\alpha,\alpha-m+1}(-\lambda t^{\alpha})$,\quad $t>0$.
\end{pro}
\begin{pro}[see \cite{Kilb06}]\label{capmitag}
Let $\alpha>0$ and $\lambda>0$, then we have\vspace{0.2cm}
\par\hspace{0.2cm}
$_{C}D^{\alpha}_{0,t}E_{\alpha,1}(-\lambda t^{\alpha})=-\lambda E_{\alpha,1}(-\lambda t^{\alpha})$,\quad $t>0$.
\end{pro}
\par Next, we prove a lemma that plays an important role in further analysis.
\begin{lemma}
Let $f(\cdot,t)\in L_2(\mathcal{G})$, $y^0\in L_2(\mathcal{G})$, then the solution $y_i(x,t)$ of IBVP $(\ref{edgemodel})$-$(\ref{edgeboundary})$ has the form
\begin{equation}\label{formalsol}
\begin{aligned}
y_i(x,t)=&\sum_{n=1}^{\infty}\langle y^0, \Psi_n\rangle E_{\alpha,1}(-\mu_nt^{\alpha})\psi_{n,i}(x)\\
&+ \sum_{n=1}^{\infty}\left(\int_{0}^{t}\langle f(x,\xi) ,\Psi_n\rangle(t-\xi)^{\alpha-1}E_{\alpha,\alpha}\left(-\mu_n(t-\xi)^{\alpha}\right) d\xi\right)\psi_{n,i}(x),
\end{aligned}
\end{equation}
where $\{\mu_n, \Psi_n\}_{n\in \mathbb{N}}$ is the eigensystem of $(\ref{eigproblem})$-$(\ref{kirchoff})$ and $\langle \cdot,\cdot \rangle$ denotes the inner product in $L_2(\mathcal{G})$.
\end{lemma}
\begin{proof}
We will use the method of eigenfunction expansions for the solution of $(\ref{edgemodel})$, that is we write the solution in terms of fourier series whose coefficients are vary with time. Hence, we try to find a solution of the equation $(\ref{graphmodel})$ in the form
\[y(x,t)=\sum_{n=1}^{\infty}T_n(t)\Psi_n(x),\]
which gives\begin{equation}\label{solui}
y_i(x,t)=\sum_{n=1}^{\infty}T_n(t)\psi_{n,i}(x),\quad i=1,2,\dotsc,k.
\end{equation}
\par Now 
\begin{align*}
&\frac{\partial^{2}y_i(x,t)}{\partial{x}^{2}}=\sum_{n=1}^{\infty}T_n(t)\psi^{''}_{n,i}(x),\\
&_{C}D^{\alpha}_{0,t}y_i(x,t)=\sum_{n=1}^{\infty}\left(_{C}D^{\alpha}_{0,t}T_n(t)\right)\psi_{n,i}(x).
\end{align*}
After substituting the value of above expressions in equation $(\ref{edgemodel})$, we get
\[\sum_{n=1}^{\infty}\left[\left(_{C}D^{\alpha}_{0,t}T_n(t)\right)\psi_{n,i}(x)-T_n(t)\psi^{''}_{n,i}(x)\right]=f_i(x,t)\]
and
\begin{equation}\label{valfk}
\sum_{n=1}^{\infty}\left[_{C}D^{\alpha}_{0,t}T_n(t)+\mu_nT_n(t)\right]\psi_{n,i}(x)=f_i(x,t),
\end{equation}
where we used the fact that $\psi^{''}_{n,i}(x)=-\mu_n\psi_{n,i}(x)$. Now we expand $y^0(x)$ and $f(x,t)$ in terms of fourier series,
\begin{equation}\label{expansion1}
f(x,t)=\sum_{n=1}^{\infty}f_n(t)\Psi_n(x)
\end{equation}
and 
\begin{equation}\label{expansion2}
y^0(x)=\sum_{n=1}^{\infty}a_n\Psi_n(x),
\end{equation}
which gives
\begin{equation}\label{iniexp}
 f_i(x,t)=\displaystyle\sum_{n=1}^{\infty}f_n(t)\psi_{n,i}(x)\quad \text{and}\quad y^0_{i}(x)=\displaystyle\sum_{n=1}^{\infty}a_n\psi_{n,i}(x),
\end{equation}
where \[f_n(t)=\langle f(x,t), \Psi_n(x)\rangle\quad \text{and} \quad a_n=\langle y^{0}(x), \Psi_n(x)\rangle.\] Hence, from equations $(\ref{valfk})$ and $(\ref{expansion1})$, we obtain
\[\sum_{n=1}^{\infty}\left[_{C}D^{\alpha}_{0,t}T_n(t)+\mu_nT_n(t)\right]\Psi_{n}(x)=\sum_{n=1}^{\infty}f_n(t)\Psi_n(x).\]
Using the uniqueness of Fourier series we get the family of fractional ODE's
\begin{equation}\label{famode}
_{C}D^{\alpha}_{0,t}T_n(t)+\mu_nT_n(t)=f_n(t)
\end{equation}
and
\[y_i(x,0)=\displaystyle\sum_{n=1}^{\infty}T_n(0)\psi_{n,i}(x)=y^0_{i}(x)=\displaystyle\sum_{n=1}^{\infty}a_n\psi_{n,i}(x),\]
 so that 
 \begin{equation}\label{initial}
 T_n(0)=a_n\,\quad n\geq 1.
 \end{equation}
 The solution of fractional differential equation $(\ref{famode})$ subject to initial condition $(\ref{initial})$ is given by \cite{Kilb06}
 \[T_n(t)=a_nE_{\alpha,1}(-\mu_nt^{\alpha})+\int_{0}^{t}(t-\xi)^{\alpha-1}E_{\alpha,\alpha}\left(-\mu_n(t-\xi)^{\alpha}\right)f_n(\xi)d\xi.\]
 Hence, from equation $(\ref{solui})$, we have
 \begin{equation}\label{solution}
\begin{aligned}
 y_i(x,t)=&\sum_{n=1}^{\infty}a_nE_{\alpha,1}(-\mu_nt^{\alpha})\psi_{n,i}(x)\\
&+\sum_{n=1}^{\infty}\left(\int_{0}^{t}(t-\xi)^{\alpha-1}E_{\alpha,\alpha}\left(-\mu_n(t-\xi)^{\alpha}\right)f_n(\xi)d\xi\right)\psi_{n,i}(x).
\end{aligned}
\end{equation}
After substituting the value of $a_n$ and $f_n(t)$ in equation $(\ref{solution})$, we get the desired result.
\end{proof}
\section{Existence and uniqueness results of a weak solution}
In this section, the existence and uniqueness of weak solutions will be proved. So let us first define the weak solution as follows.
\begin{definition}\label{weaksoln}
We define $y$ as a weak solution of $(\ref{graphmodel})$-$(\ref{graphini})$ if $(\ref{graphmodel})$ holds in $L_2(\mathcal{G})$ and $y(\cdot,t) \in V$ for almost all $t\in (0,T)$ 
and satisfy\\
\par\hspace{0.2cm}
$\displaystyle\lim_{t\to 0}~ \norm{y(\cdot,t)-y^0}_{L_2(\mathcal{G})}=0$.
\end{definition}
\par Now we state our first main result as follows.
\begin{theorem}\label{main}
Let $y^{0}\in L_2(\mathcal{G})$ and $f(x,t)\in L_{\infty}(0,T; L_2(\mathcal{G}))$. Then there exists a unique weak\vspace{0.1cm}\\ solution $y \in C([0,T]; L_2(\mathcal{G})) \cap C((0,T]; D(\mathcal{-L}))$ such that $_{C}D^{\alpha}_{0,t}u \in L_{\infty}(0,T; L_2(\mathcal{G}))$. Furthermore,\vspace{0.1cm}\\ there exists a positive constant $C_1$ such that\\
\begin{equation}\label{thm1bound1}
\norm{y}_{C([0,T]; L_2(\mathcal{G}))}\leq C_1\left(\norm{y^{0}}_{L_2(\mathcal{G})}+\norm{f}_{L_{\infty}(0,T; L_2(\mathcal{G}))}\right),
\end{equation}
\begin{equation}\label{thm1bound2}
\norm{y(\cdot,t)}_{\prod_{i=1}^{k}H_2(0,l_i)}\leq C_1\left(\norm{y^{0}}_{L_2(\mathcal{G})}t^{-\alpha}+\norm{f}_{L_{\infty}(0,T; L_2(\mathcal{G}))}\right).
\end{equation}
\end{theorem}
\begin{proof}
We will show that $y(x,t)=(y_i(x,t))_{i=1}^{k}$, where $y_i(x,t)$ is given by equation $(\ref{formalsol})$, certainly gives the weak solution to $(\ref{graphmodel})$-$(\ref{graphini})$. We assume $C>0$ to be a generic constant in the following proof. Hence, using equation $(\ref{formalsol})$ and the fact that\\ $\displaystyle\sum_{i=1}^{k}\langle \psi_{n,i}, \psi_{m,i}\rangle_{L_2(0,l_i)}=\begin{cases}
1 & \text{if } m=n\\
0 & \text{for } m\neq n
\end{cases}
$, we have
\begin{align*}
\sum_{i=1}^{k}\norm{y_i(\cdot, t)}_{L_2(0,l_i)}^{2}=&\sum_{n=1}^{\infty}\left|\langle y^0, \Psi_n\rangle E_{\alpha,1}(-\mu_nt^{\alpha})\right|^{2}\\
&+\sum_{n=1}^{\infty}\left|\int_{0}^{t}\langle f(\cdot,\xi) ,\Psi_n\rangle(t-\xi)^{\alpha-1}E_{\alpha,\alpha}\left(-\mu_n(t-\xi)^{\alpha}\right) d\xi\right|^{2}.
\end{align*}
Using Propositions $\ref{monotone2}$ and $\ref{monotone1}$, we get
\begin{align*}
\sum_{i=1}^{k}\norm{y_i(\cdot, t)}_{L_2(0,l_i)}^{2}&\leq\sum_{n=1}^{\infty}\left|\langle y^0, \Psi_n\rangle\right|^{2}+\sum_{n=1}^{\infty}\left|\int_{0}^{t}\langle f(\cdot,\xi) ,\Psi_n\rangle\frac{(t-\xi)^{\alpha-1}}{\Gamma{(\alpha)}}d\xi\right|^{2}\\
&\leq \|y^0\|_{L_2(\mathcal{G})}^{2}+\sum_{n=1}^{\infty}\sup_{0\leq t\leq T}\left|\langle f(\cdot,t) ,\Psi_n\rangle\right|^{2}\left(\frac{t^{\alpha}}{\Gamma{(\alpha+1)}}\right)^{2}\\
 &\leq \|y^0\|_{L_2(\mathcal{G})}^{2}+\norm{f}_{L_{\infty}(0,T; L_2(\mathcal{G}))}^2\frac{T^{2\alpha}}{(\Gamma{(\alpha+1)})^2}\cdot
\end{align*}
Hence,\\
\par\hspace{0.2cm}
$\norm{y(\cdot,t)}_{L_2(\mathcal{G})}\leq C_1\left(\norm{y^{0}}_{L_2(\mathcal{G})}+\norm{f}_{L_{\infty}(0,T; L_2(\mathcal{G}))}\right)$,\quad $t\in [0,T]$,\vspace{0.2cm}\\
where $C_1=max\left\{1,\frac{T^{\alpha}}{\Gamma{(\alpha+1)}}\right\}$.\\
\par\hspace{0.2cm}
Now, given $t$, $t+h\in [0,T]$, we have
\begin{align*}
y_i(x,t+h)-y_i(x,t)=&\sum_{n=1}^{\infty}\langle y^0, \Psi_n\rangle \left(E_{\alpha,1}(-\mu_n(t+h)^{\alpha})-E_{\alpha,1}(-\mu_nt^{\alpha})\right)\psi_{n,i}(x)\\
&+ \sum_{n=1}^{\infty}(u_n(t+h)-u_n(t))\psi_{n,i}(x),
\end{align*}
where \[u_n(t)=\int_{0}^{t}\langle f(x,\xi) ,\Psi_n\rangle(t-\xi)^{\alpha-1}E_{\alpha,\alpha}\left(-\mu_n(t-\xi)^{\alpha}\right) d\xi.\]
\begin{align*}
\sum_{i=1}^{k}\norm{y_i(\cdot,t+h)-y_i(\cdot,t)}_{L_2(\mathcal{G})}^{2}=&\sum_{n=1}^{\infty}\left|\langle y^0, \Psi_n\rangle \left(E_{\alpha,1}(-\mu_n(t+h)^{\alpha})-E_{\alpha,1}(-\mu_nt^{\alpha})\right)\right|^{2}\\
&+ \sum_{n=1}^{\infty}\left|(u_n(t+h)-u_n(t))\right|^{2}.
\end{align*}
Again, using Propositions $\ref{monotone2}$ and $\ref{monotone1}$, we get
\[\sum_{i=1}^{k}\norm{y_i(\cdot,t+h)-y_i(\cdot,t)}_{L_2(0,l_i)}^{2}\leq 4\|y^{0}\|_{L_2(\mathcal{G})}^{2}+C\norm{f}_{L_{\infty}(0,T; L_2(\mathcal{G}))}.\]
Since $\displaystyle\lim_{h\to 0}~ \left|E_{\alpha,1}(-\mu_n(t+h)^{\alpha})-E_{\alpha,1}(-\mu_nt^{\alpha})\right|=0$, $\displaystyle\lim_{h\to 0}~ \left|u_n(t+h)-u_n(t)\right|=0$ (see Lemma 2.14 \cite{Xl16}). Hence, using Lebesgue dominated convergence theorem, we get
\[\displaystyle\lim_{h \to 0}~\norm{y(\cdot,t+h)-y(\cdot,t)}_{L_2(\mathcal{G})}^{2}=\displaystyle\lim_{h \to 0}~\left(\sum_{i=1}^{k}\norm{y_i(\cdot,t+h)-y_i(\cdot,t)}_{L_2(0,l_i)}^{2}\right)=0.\]
Therefore, $y \in C([0,T]; L_2(\mathcal{G}))$.\\
\par Now it will be shown that $y\in C((0,T]; D(\mathcal{-L}))$ and $_{C}D^{\alpha}_{0,t}y \in L_{\infty}(0,T; L_2(\mathcal{G}))$. We have
\begin{align*}
(-\mathcal{M})y_i(x,t)=&\sum_{n=1}^{\infty}\mu_n\langle y^0, \Psi_n\rangle E_{\alpha,1}(-\mu_nt^{\alpha})\psi_{n,i}(x)\\
&+\sum_{n=1}^{\infty}\mu_n\left(\int_{0}^{t}\langle f(\cdot,\xi) ,\Psi_n\rangle(t-\xi)^{\alpha-1}E_{\alpha,\alpha}\left(-\mu_n(t-\xi)^{\alpha}\right)d\xi\right)\psi_{n,i}(x).
\end{align*}
Now, \begin{align*}
\norm{(\mathcal{-L})y(\cdot,t)}_{L_2(\mathcal{G})}^{2}=&\sum_{i=1}^{k}\norm{(-\mathcal{M})y_i(\cdot,t)}_{L_2(0,l_i)}^{2}\\
=&\sum_{n=1}^{\infty}\mu_n^{2}\left|\langle y^0, \Psi_n\rangle E_{\alpha,1}(-\mu_nt^{\alpha})\right|^{2}\\
&+\sum_{n=1}^{\infty}\mu_n^{2}\left|\int_{0}^{t}\langle f(\cdot,\xi) ,\Psi_n\rangle(t-\xi)^{\alpha-1}E_{\alpha,\alpha}\left(-\mu_n(t-\xi)^{\alpha}\right) d\xi\right|^{2}.
\end{align*}
Also from Propositions $\ref{monotone1}$ and $\ref{dermit}$
\begin{equation}\label{convbound}
\begin{aligned}
\int_{0}^{t}\left|\xi^{\alpha-1}E_{\alpha, \alpha}(-\mu_n\xi^{\alpha})\right|d\xi&=\int_{0}^{t}\xi^{\alpha-1}E_{\alpha, \alpha}(-\mu_n\xi^{\alpha})d\xi\\
&=-\frac{1}{\mu_n}\int_{0}^{t}\frac{d}{d\xi}E_{\alpha,1}(-\mu_n\xi^{\alpha})d\xi=\frac{1}{\mu_n}\left(1-E_{\alpha,1}(-\mu_nt^{\alpha})\right)\leq \frac{1}{\mu_n}.
\end{aligned}
\end{equation}
Now, using equation $(\ref{convbound})$, Proposition $\ref{boundmit}$ and Young inequality for the convolution, we get
\begin{align*}
\norm{(\mathcal{-L})y}_{L_2(\mathcal{G})}^{2}&\leq \sum_{n=1}^{\infty}\mu_n^{2}\left|\langle y^0, \Psi_n\rangle\right|^{2}\left(\frac{C_1}{1+\mu_nt^{\alpha}}\right)^{2}\\
&+\sum_{n=1}^{\infty}\mu_n^{2}\sup_{0\leq t\leq T}\left|\langle f(\cdot,t) ,\Psi_n\rangle\right|^{2}\left|\int_{0}^{T}t^{\alpha-1}E_{\alpha,\alpha}(-\mu_nt^{\alpha})dt\right|^{2}.
\end{align*}

Hence, we obtain
\begin{equation}\label{lubound}
\norm{(\mathcal{-L})y}_{L_2(\mathcal{G})}^{2}\leq\|y^{0}\|_{L_2(\mathcal{G})}^2t^{-2\alpha}+\norm{f}_{L_{\infty}(0,T; L_2(\mathcal{G}))}^{2}.
\end{equation}
Since $\mathcal{-L}y$ is convergent in $L_2(\mathcal{G})$ uniformly on $t\in (t_0,T]$ with any given $t_0>0$, we deduce that \vspace{0.1cm}\\ $\mathcal{-L}y\in C((0,T]; L_2(\mathcal{G}))$, that is $-\mathcal{M}y_i\in C((0,T]; L_2(0,l_i))$, $i=1,2\dotsc,k$ and hence $y\in C((0,T]; D(\mathcal{-L}))$. Moreover, we obtain the following estimate from equation $(\ref{lubound})$
\begin{align*}
\|y(\cdot,t)\|_{\prod_{i=1}^{k}H_2(0,l_i)}&= \sum_{i=1}^{k}\norm{y_i(\cdot,t)}_{H_2(0,l_i)}\\
&\leq C' \sum_{i=1}^{k}\norm{(\mathcal{-L})y_i(\cdot,t)}_{L_2(0,l_i)}\\
&=C'\norm{(\mathcal{-L})y_i(\cdot,t)}_{L_2(\mathcal{G})}\leq C\left(\|y^{0}\|_{L_2(\mathcal{G})}t^{-\alpha}+\norm{f}_{L_{\infty}(0,T; L_2(\mathcal{G}))}\right).
\end{align*}
By $(\ref{graphmodel})$, we see that $_{C}D^{\alpha}_{0,t}y \in L_{\infty}(0,T;L_2(\mathcal{G}))$ and $(\ref{graphmodel})$ holds in $L_2(\mathcal{G})$ for $t \in (0,T]$.\\
\par Next it will be shown that $\displaystyle\lim_{t\to 0}~\norm{y(\cdot,t)-y^0}_{L_2(\mathcal{G})}=0$. From equations $(\ref{formalsol})$ and $(\ref{iniexp})$, we have
\begin{align*}
y_i(x,t)-y^0_i(x)=&\sum_{n=1}^{\infty}\langle y^0, \Psi_n\rangle \left(E_{\alpha,1}(-\mu_nt^{\alpha})-1\right)\psi_{n,i}(x)\\
&+ \sum_{n=1}^{\infty}\left(\int_{0}^{t}\langle f(x,\xi) ,\Psi_n\rangle(t-\xi)^{\alpha-1}E_{\alpha,\alpha}\left(-\mu_n(t-\xi)^{\alpha}\right) d\xi\right)\psi_{n,i}(x).
\end{align*}
Hence,
\begin{align*}
\sum_{i=1}^{k}\norm{y_i(\cdot,t)-y^0_i(\cdot)}_{L_2(0,l_i)}^{2}&\leq \sum_{n=1}^{\infty}\left|\langle y^0, \Psi_n\rangle \left(E_{\alpha,1}(-\mu_nt^{\alpha})-1\right)\right|^{2}\\
&+ \sum_{n=1}^{\infty}\left|\int_{0}^{t}\langle f(x,\xi) ,\Psi_n\rangle(t-\xi)^{\alpha-1}E_{\alpha,\alpha}\left(-\mu_n(t-\xi)^{\alpha}\right) d\xi\right|^{2}\\
&=: V_1(t)+V_2(t).
\end{align*}
Clearly, $\displaystyle\lim_{t\to 0}~V_2(t)=0$, using Proposition $\ref{monotone1}$
\[V_1(t)=\sum_{n=1}^{\infty}\left|\langle y^0, \Psi_n\rangle \left(E_{\alpha,1}(-\mu_nt^{\alpha})-1\right)\right|^{2}\leq C\|y^0\|_{L_2(\mathcal{G})}^{2}\]
and $\displaystyle\lim_{t\to 0}~ \left(E_{\alpha,1}(-\mu_nt^{\alpha})-1\right)=0$. Hence, by using Lebesgue dominated convergence theorem, we have $\displaystyle\lim_{t\to 0}~ V_1(t)=0$. Hence $\displaystyle\lim_{t\to 0}~\sum_{i=1}^{k}\norm{y_i(\cdot,t)-y^0_i(\cdot)}_{L_2(0,l_i)}=0$, which shows that \[\displaystyle\lim_{t\to 0}~\norm{y(\cdot,t)-y^0}_{L_2(\mathcal{G})}=0.\]
 Finally, we show the uniqueness of the weak solution to initial-value problem $(\ref{graphmodel})$-$(\ref{graphini})$.\vspace{0.2cm}
 \par \textbf{Uniqueness:} Under the conditions $y^0=0$ and $f=0$, we need to show that system $(\ref{edgemodel})$-$(\ref{edgeboundary})$ has only the trivial solution. On taking the inner product of $(\ref{graphmodel})$ with $\Psi_n(x)$, applying Green's formula and setting $y^{n}(t)=(y(\cdot, t), \Psi_n)$, we obtain
 \begin{align*}
 _{C}D^{\alpha}_{0,t}y^{n}(t)=&\int_{\mathcal{G}}\frac{\partial^2y(x,t)}{\partial x^2}\Psi_n(x)dx\\
=&\sum_{i=1}^{k}\int_{0}^{l_i}\frac{\partial^2y_i(x,t)}{\partial x^2}\psi_{n,i}(x)dx\\
=&-\sum_{i=1}^{k}\int_{0}^{l_i}\frac{\partial y_i(x,t)}{\partial x}\psi'_{n,i}(x)dx+\sum_{i=1}^{k}\frac{\partial y_i(x,t)}{\partial x}\psi_{n,i}(x)\bigg|_{0}^{l_i}.
 \end{align*}
 Using equations $(\ref{edgekirchoff})$ and $(\ref{continuity})$, we get
 \begin{align*}
 \sum_{i=1}^{k}\frac{\partial y_i(x,t)}{\partial x}\psi_{n,i}(x)\bigg|_{0}^{l_i}=&\sum_{i=1}^{k}\frac{\partial y_i(l_i,t)}{\partial x}\psi_{n,i}(l_i)-\sum_{i=1}^{k}\frac{\partial y_i(0,t)}{\partial x}\psi_{n,i}(0)\\
 =&-\sum_{i=1}^{k}\frac{\partial y_i(0,t)}{\partial x}\psi_{n,i}(0)=-\phi_n(0)\sum_{i=1}^{k}\frac{\partial y_i(0,t)}{\partial x}=0,
 \end{align*}
 where $\psi_{n,i}(0)=\psi_{n,j}(0)=\phi_n(0)$,~$i\neq j$,~ $i,j=1,2,\dotsc,k$. Hence we get
 \begin{align*}
  _{C}D^{\alpha}_{0,t}y^{n}(t)=&-\sum_{i=1}^{k}\int_{0}^{l_i}\frac{\partial y_i(x,t)}{\partial x}\psi'_{n,i}(x)dx\\
  =&\sum_{i=1}^{k}\int_{0}^{l_i}y_i(x,t)\psi^{''}_{n,i}(x) dx-\sum_{i=1}^{k}y_i(x,t)\psi'_{n,i}(x)\bigg|_{0}^{l_i}.
  \end{align*}
  Again using equations $(\ref{edgecontinuity})$ and $(\ref{kirchoff})$ and a similar approach as above, we get \[\sum_{i=1}^{k}y_i(x,t)\psi'_{n,i}(x)\bigg|_{0}^{l_i}=0.\] Therefore,
  \begin{align*}
  _{C}D^{\alpha}_{0,t}y^{n}(t)=&\sum_{i=1}^{k}\int_{0}^{l_i}y_i(x,t)\psi^{''}_{n,i}(x)dx\\
  =&-\mu_n\sum_{i=1}^{k}\int_{0}^{l_i}y_i(x,t)\psi_{n,i}(x)dx=-\mu_n\langle u(\cdot, t),\Psi_n\rangle.
  \end{align*}
  Hence, we obtain the following initial value fractional differential equation
\par\hspace{0.2cm} 
$\qquad \begin{cases}
_{C}D^{\alpha}_{0,t}y^{n}(t)=-\mu_ny^{n}(t),\quad \quad t\in (0,T),\\
 y^{n}(0)=0.
\end{cases}$\vspace{0.1cm}\\
Due to the existence and uniqueness of the above fractional differential equation, we get that $y^{n}(t)=0,~n=1,2,\cdots$. Since $\Psi_n$ is a complete orthonormal basis in $L_2(\mathcal{G})$, we have $y=0$ in $\mathcal{G}\times (0,T]$.
\end{proof}
\begin{theorem}\label{u0invthm}
Let $y^0\in V$, $f(x,t)\in L_{\infty}(0,T; L_2(\mathcal{G}))$. Then there exists a unique weak\vspace{0.1cm} solution $y \in L_2((0,T]; D(\mathcal{-L}))$ such that $_{C}D^{\alpha}_{0,t}y \in L_{2}(\mathcal{G}\times(0,T))$ and the following inequalty holds:
\begin{equation}\label{y^0invbound}
\norm{y}_{L_2\left((0,T]; \prod_{i=1}^{k}H_2(0,l_i)\right)}+\norm{_{C}D^{\alpha}_{0,t}y}_{L_{2}(\mathcal{G}\times(0,T))}\leq C\left(\norm{y^{0}}_{V}+\norm{f}_{L_{\infty}(0,T; L_2(\mathcal{G}))}\right).
\end{equation}
\end{theorem}
\begin{proof}
\begin{align*}
\norm{(\mathcal{-L})y(\cdot,t)}_{L_2(\mathcal{G})}^{2}&=\sum_{i=1}^{k}\norm{(-\mathcal{M})y_i(\cdot,t)}_{L_2(0,l_i)}^{2}\\
&=\sum_{n=1}^{\infty}\left|\mu_n\langle y^0, \Psi_n\rangle E_{\alpha,1}(-\mu_nt^{\alpha})\right|^{2}\\
&\hspace{0.3cm}+\sum_{n=1}^{\infty}\mu_n^{2}\left|\int_{0}^{t}\langle f(\cdot,\xi) ,\Psi_n\rangle(t-\xi)^{\alpha-1}E_{\alpha,\alpha}\left(-\mu_n(t-\xi)^{\alpha}\right) d\xi\right|^{2}.
\end{align*}
Now, using Proposition $\ref{boundmit}$ and Young inequality for the convolution, we get
\begin{align*}
\norm{(\mathcal{-L})y(\cdot,t)}_{L_2(\mathcal{G})}^{2}&\leq \sum_{n=1}^{\infty}\mu_n\left|\langle y^0, \Psi_n\rangle\right|^{2}\left(\frac{C_1\sqrt{\mu_n}}{1+\mu_nt^{\alpha}}\right)^{2}\\
&\hspace{0.2cm}+\sum_{n=1}^{\infty}\mu_n^{2}\sup_{0\leq t\leq T}\left|\langle f(\cdot,t) ,\Psi_n\rangle\right|^{2}\left|\int_{0}^{T}t^{\alpha-1}E_{\alpha,\alpha}(-\mu_nt^{\alpha})dt\right|^{2}\\
&=\sum_{n=1}^{\infty}\mu_n\left|\langle y^0, \Psi_n\rangle\right|^{2}\left(\frac{C_1\sqrt{\mu_nt^{\alpha}}}{1+\mu_nt^{\alpha}}\right)^{2}t^{-\alpha}\\
&\hspace{0.2cm}+\sum_{n=1}^{\infty}\mu_n^{2}\sup_{0\leq t\leq T}\left|\langle f(\cdot,t) ,\Psi_n\rangle\right|^{2}\left|\int_{0}^{T}t^{\alpha-1}E_{\alpha,\alpha}(-\mu_nt^{\alpha})dt\right|^{2}\\
&\leq C\norm{y^0}_{V}^2t^{-\alpha}+\norm{f}_{L_{\infty}(0,T; L_2(\mathcal{G}))}^2,
\end{align*}
where we have used equation $(\ref{convbound})$. 
\par Now, \begin{align*}
\norm{y}_{L_2\left((0,T]; \prod_{i=1}^{k}H_2(0,l_i)\right)}^2&=\int_{0}^{T}\norm{ y(\cdot,t)}_{\prod_{i=1}^{k}H_2(0,l_i)}^2dt\\
&\leq \int_{0}^{T}\left(C\norm{y^0}_{V}^2t^{-\alpha}+\norm{f}_{L_{\infty}(0,T; L_2(\mathcal{G}))}^2\right)dt\\
&=\frac{CT^{1-\alpha}}{1-\alpha}\norm{y^0}_{V}^2+T\norm{f}_{L_{\infty}(0,T; L_2(\mathcal{G}))}^2\\
&\leq C_1\left(\norm{y^{0}}_{V}^2+\norm{f}_{L_{\infty}(0,T; L_2(\mathcal{G}))}^2\right).
\end{align*}
Therefore, we have $y \in L_2((0,T]; D(\mathcal{-L}))$.
\par Now, using Proposition $\ref{capmitag}$ and Lemma 2.8 in \cite{Wei18}, we have
\begin{align*}
_{C}D^{\alpha}_{0,t}y_i(x,t)=&-\sum_{n=1}^{\infty}\mu_n\langle y^0, \Psi_n\rangle E_{\alpha,1}(-\mu_nt^{\alpha})\psi_{n,i}(x)+\sum_{n=1}^{\infty}\langle f(x,t),\Psi_n\rangle\psi_{n,i}(x)\\
&- \sum_{n=1}^{\infty}\mu_n\left(\int_{0}^{t}\langle f(x,\xi) ,\Psi_n\rangle(t-\xi)^{\alpha-1}E_{\alpha,\alpha}\left(-\mu_n(t-\xi)^{\alpha}\right) d\xi\right)\psi_{n,i}(x).
\end{align*}
Hence,
\begin{align*}
\sum_{i=1}^{k}\norm{_{C}D^{\alpha}_{0,t}y_i(\cdot,t)}_{L_2(0,l_i)}^{2}&\leq \sum_{n=1}^{\infty}\left|\mu_n\langle y^0, \Psi_n\rangle E_{\alpha,1}(-\mu_nt^{\alpha})\right|^{2}+\sum_{n=1}^{\infty}\left|\langle f(x,t),\Psi_n\rangle\right|^{2}\\
&+\sum_{n=1}^{\infty}\mu_n^{2}\left|\int_{0}^{t}\langle f(x,\xi) ,\Psi_n\rangle(t-\xi)^{\alpha-1}E_{\alpha,\alpha}\left(-\mu_n(t-\xi)^{\alpha}\right) d\xi\right|^{2}\\
&\leq \sum_{n=1}^{\infty}\mu_n\left|\langle y^0, \Psi_n\rangle\right|^{2}\left(\frac{C_1\sqrt{\mu_nt^{\alpha}}}{1+\mu_nt^{\alpha}}\right)^{2}t^{-\alpha}+\sum_{n=1}^{\infty}\left|\langle f(x,t),\Psi_n\rangle\right|^{2}\\
&+\sum_{n=1}^{\infty}\mu_n^{2}\sup_{0\leq t\leq T}\left|\langle f(\cdot,t) ,\Psi_n\rangle\right|^{2}\left|\int_{0}^{T}t^{\alpha-1}E_{\alpha,\alpha}(-\mu_nt^{\alpha})dt\right|^{2}.
\end{align*}
Again, using equation $(\ref{convbound})$, we obtain
\begin{align*}
\norm{_{C}D^{\alpha}_{0,t}y(\cdot,t)}_{L_2(\mathcal{G})}^{2}=&\sum_{i=1}^{k}\norm{_{C}D^{\alpha}_{0,t}y_i(\cdot,t)}_{L_2(0,l_i)}^{2}\\
&\leq C \|y^0\|_{V}^{2}t^{-\alpha}+\norm{f(\cdot,t)}_{L_2(\mathcal{G})}^2+\norm{f}_{L_{\infty}(0,T; L_2(\mathcal{G}))}^{2}\\
&\leq C_1\left(\|y^0\|_{V}^{2}t^{-\alpha}+\norm{f}_{L_{\infty}(0,T; L_2(\mathcal{G}))}^{2}\right).
\end{align*}
Since $0<\alpha<1$, we see that \[\norm{_{C}D^{\alpha}_{0,t}y}_{L_{2}(\mathcal{G}\times(0,T))}\leq C\left(\|y^0\|_{V}+\norm{f}_{L_{\infty}(0,T; L_2(\mathcal{G}))}\right).\] Therefore, we have $_{C}D^{\alpha}_{0,t}y\in L_{2}(\mathcal{G}\times(0,T))$. The proof of $\displaystyle\lim_{t\to 0}~ \norm{y(\cdot,t)-y^0}_{L_2(\mathcal{G})}=0$ and uniqueness of weak solution is similar to the one derived in the proof of Theorem $\ref{main}$. Thus the proof of Theorem $\ref{u0invthm}$ is complete.
\end{proof}
\begin{theorem}\label{main2}
Let $y^0\in D(-\mathcal{L})$, $f(x,t)\in L_{\infty}(0,T; L_2(\mathcal{G}))$. Then there exists a unique weak\vspace{0.1cm}\\ solution $y \in C([0,T]; L_2(\mathcal{G})) \cap C((0,T]; D(\mathcal{-L}))$ such that $_{C}D^{\alpha}_{0,t}u \in L_{2}(\mathcal{G}\times(0,T))$. Moreover\vspace{0.1cm}\\ there exists a constant $C_1>0$ such that
\begin{equation}\label{thm2bound}
\norm{y}_{C\left([0,T]; \prod_{i=1}^{k}H_2(0,l_i)\right)}+\norm{_{C}D^{\alpha}_{0,t}y}_{L_{2}(\mathcal{G}\times(0,T))}\leq C_1\left(\norm{y^{0}}_{\prod_{i=1}^{k}H_2(0,l_i)}+\norm{f}_{L_{\infty}(0,T; L_2(\mathcal{G}))}\right).
\end{equation}
\end{theorem}
\begin{proof}
Under the assumption $y^0\in D(\mathcal{-L})$, using Proposition $\ref{dermit}$, Proposition $\ref{monotone1}$ and Young inequality for the convolution, we have	
 \begin{align*}
\norm{(\mathcal{-L})y(\cdot,t)}_{L_2(\mathcal{G})}^{2}&=\sum_{i=1}^{k}\norm{(-\mathcal{M})y_i(\cdot,t)}_{L_2(0,l_i)}^{2}\\
&=\sum_{n=1}^{\infty}\mu_n^{2}\left|\langle y^0, \Psi_n\rangle E_{\alpha,1}(-\mu_nt^{\alpha})\right|^{2}\\
&\hspace{0.3cm}+\sum_{n=1}^{\infty}\mu_n^{2}\left|\int_{0}^{t}\langle f(\cdot,\xi) ,\Psi_n\rangle(t-\xi)^{\alpha-1}E_{\alpha,\alpha}\left(-\mu_n(t-\xi)^{\alpha}\right) d\xi\right|^{2}\\
&\leq\sum_{n=1}^{\infty}\mu_n^{2}\left|\langle y^0, \Psi_n\rangle\right|^{2}\\
&\hspace{0.3cm}+\sum_{n=1}^{\infty}\mu_n^{2}\sup_{0\leq t\leq T}\left|\langle f(\cdot,t) ,\Psi_n\rangle\right|^{2}\left|\int_{0}^{t}t^{\alpha-1}E_{\alpha,\alpha}(-\mu_nt^{\alpha})dt\right|^{2}\\
&\leq \|y^0\|_{\prod_{i=1}^{k}H_2(0,l_i)}^{2}+\norm{f}_{L_{\infty}(0,T; L_2(\mathcal{G}))}^{2}.
\end{align*}
Hence, we obtain
\begin{align*}
\|y(\cdot,t)\|_{\prod_{i=1}^{k}H_2(0,l_i)}&\leq C'\norm{(\mathcal{-L})y(\cdot,t)}_{L_2(\mathcal{G})}\\
&\leq C\left(\|y^{0}\|_{\prod_{i=1}^{k}H_2(0,l_i)}+\norm{f}_{L_{\infty}(0,T; L_2(\mathcal{G}))}\right).
\end{align*}
Now,
\begin{align*}
\sum_{i=1}^{k}\norm{_{C}D^{\alpha}_{0,t}y_i(\cdot,t)}_{L_2(0,l_i)}^{2}&\leq \sum_{n=1}^{\infty}\mu_n^{2}\left|\langle y^0, \Psi_n\rangle E_{\alpha,1}(-\mu_nt^{\alpha})\right|^{2}+\sum_{n=1}^{\infty}\left|\langle f(x,t),\Psi_n\rangle\right|^{2}\\
&+\sum_{n=1}^{\infty}\mu_n^{2}\left|\int_{0}^{t}\langle f(x,\xi) ,\Psi_n\rangle(t-\xi)^{\alpha-1}E_{\alpha,\alpha}\left(-\mu_n(t-\xi)^{\alpha}\right) d\xi\right|^{2}\\
&\leq \sum_{n=1}^{\infty}\mu_n^{2}\left|\langle y^0, \Psi_n\rangle\right|^{2}+\sum_{n=1}^{\infty}\left|\langle f(x,t),\Psi_n\rangle\right|^{2}\\
&+\sum_{n=1}^{\infty}\mu_n^{2}\sup_{0\leq t\leq T}\left|\langle f(\cdot,t) ,\Psi_n\rangle\right|^{2}\left|\int_{0}^{t}t^{\alpha-1}E_{\alpha,\alpha}(-\mu_nt^{\alpha})dt\right|^{2}.
\end{align*}
Using equation $(\ref{convbound})$, we obtain
\begin{align*}
\norm{_{C}D^{\alpha}_{0,t}y(\cdot,t)}_{L_2(\mathcal{G})}^{2}=&\sum_{i=1}^{k}\norm{_{C}D^{\alpha}_{0,t}y_i(\cdot,t)}_{L_2(0,l_i)}^{2}\\
&\leq\|y^0\|_{\prod_{i=1}^{k}H_2(0,l_i)}^{2}+\norm{f(\cdot,t)}_{L_2(\mathcal{G})}^2+\norm{f}_{L_{\infty}(0,T; L_2(\mathcal{G}))}^{2}\vspace{0.1cm}\\
&\leq \|y^0\|_{\prod_{i=1}^{k}H_2(0,l_i)}^{2}+C\norm{f}_{L_{\infty}(0,T; L_2(\mathcal{G}))}^{2}.
\end{align*}
Hence,
\begin{align*}
\norm{_{C}D^{\alpha}_{0,t}y}_{{L_{2}(\mathcal{G}\times(0,T))}}^2=&\int_{0}^{T}\norm{_{C}D^{\alpha}_{0,t}y(\cdot,t)}_{L_2(\mathcal{G})}^{2}dt\\
&\leq T\left(\|y^0\|_{\prod_{i=1}^{k}H_2(0,l_i)}^{2}+C\norm{f}_{L_{\infty}(0,T; L_2(\mathcal{G}))}^{2}\right)\\
&\leq C_1\left(\|y^0\|_{\prod_{i=1}^{k}H_2(0,l_i)}^{2}+\norm{f}_{L_{\infty}(0,T; L_2(\mathcal{G}))}^2\right).
\end{align*}
\end{proof}
\begin{corollary}
Let $y^0\in L_2(\mathcal{G})$ and $f=0$. Then we obtain the following estimate for the unique weak solution $y \in C([0,T]; L_2(\mathcal{G})) \cap C((0,T]; D(\mathcal{-L})):$
\begin{equation}\label{cor1bound1}
\norm{y(\cdot,t)}_{L_2(\mathcal{G})}\leq \frac{C}{1+\mu_1t^{\alpha}}\norm{y^0}_{L_2(\mathcal{G})},\quad t\in (0,T).
\end{equation}
\end{corollary}
\begin{proof}
Substituting $f=0$ in equation $(\ref{formalsol})$ and using Proposition $\ref{boundmit}$, we obtain
\begin{align*}
\sum_{i=1}^{k}\norm{y_i(\cdot, t)}_{L_2(0,l_i)}^{2}=&\sum_{n=1}^{\infty}\left|\langle y^0, \Psi_n\rangle E_{\alpha,1}(-\mu_nt^{\alpha})\right|^{2}\\
&\leq \sum_{n=1}^{\infty}\left|\langle y^0, \Psi_n\rangle\right|^2\left(\frac{C}{1+\mu_nt^{\alpha}}\right)^2\leq \left(\frac{C}{1+\mu_1t^{\alpha}}\right)^2\|y^0\|_{L_2(\mathcal{G})}^2,\quad t\in (0,T).
\end{align*}
\end{proof}
\begin{corollary}
Let $y^0\in D(-\mathcal{L})$ and $f=0$. Then there exists a constant $C_1>0$ such that
\begin{equation}\label{cor2bound1}
\norm{y(\cdot,t)}_{\prod_{i=1}^{k}H_2(0,l_i)}+\norm{_{C}D^{\alpha}_{0,t}u(\cdot,t)}_{L_2(\mathcal{G})} \leq \frac{C_1}{1+\mu_1t^{\alpha}}\norm{y^0}_{\prod_{i=1}^{k}H_2(0,l_i)},\quad t\in (0,T).
\end{equation}
\begin{proof}
\begin{align*}
\norm{(\mathcal{-L})y(\cdot,t)}_{L_2(\mathcal{G})}^{2}&=\sum_{i=1}^{k}\norm{(-\mathcal{M})y_i(\cdot,t)}_{L_2(0,l_i)}^{2}\\
&=\sum_{n=1}^{\infty}\mu_n^{2}\left|\langle y^0, \Psi_n\rangle E_{\alpha,1}(-\mu_nt^{\alpha})\right|^{2}\\
&\leq \left(\frac{C_1}{1+\mu_nt^{\alpha}}\right)^2\sum_{n=1}^{\infty}\mu_n^{2}\left|\langle y^0, \Psi_n\rangle\right|^2\leq \left(\frac{C_1}{1+\mu_1t^{\alpha}}\right)^2\|y^0\|_{\prod_{i=1}^{k}H_2(0,l_i)}^2,\quad t\in (0,T).
\end{align*}
Since $f=0$, equation $(\ref{graphmodel})$ implies that $_{C}D^{\alpha}_{0,t}y(\cdot,t)=\mathcal{L}y(\cdot,t)$ and hence estimate $(\ref{cor2bound1})$ follows.
\end{proof}
\end{corollary}

\section{Conclusion and future work}
In this paper, the existence and uniqueness of time-fractional diffusion equation on a star graph is established. By using the method of eigenfunction expansion the existence and uniqueness of the weak solution and the regularity of the solution is derived. In future we will consider fractional diffusion equation on more general graphs (i.e. graphs containing cycles) and investigate the existence and uniqueness of solution. We will also consider optimal control problems for fractional diffusion equation on metric graphs.
\section*{Acknowledgements}
The authors would like to thank the Indo-German exchange program ``Multiscale Modelling, Simulation and optimization for energy, Advanced Materials and Manufacturing''. The program (grant number 1-3/2016 (IC)) is funded by University Grants Commission (India) and  DAAD (Germany). The coordination of the program through the ``Central Institute for Scientific Computing'' at Friedrich-Alexander-Universit\"{a}t, Erlangen is acknowledged.
\bibliography{diffusion}
\bibliographystyle{plain}
\end{document}